\documentclass[english,12pt]{amsart}
\usepackage{amssymb}
\usepackage{amsfonts}
\usepackage{amscd}
\usepackage[T1]{fontenc}

\newcommand{\iso}{\mathrel{\overset{1:1}{\smash{\longrightarrow}\vrule height.3ex width0ex\relax}}}

\swapnumbers

\def\ac{{\overline{\rm ac}}}

\def\rv{ rv }

\def\11{{\mathbf 1}}

\def\QQ{{\mathbb Q}}
\def\RR{{\mathbb R}}

\def\ZZ{{\mathbb Z}}

\def\cL{{\mathcal L}}
\def\cM{{\mathcal M}}

\def\cO{{\mathcal O}}

\mathchardef\alphag="7C0B
\mathchardef\betag="7C0C
\mathchardef\gammag="7C0D
\mathchardef\deltag="7C0E
\mathchardef\varepsilong="7C22
\mathchardef\varphig="7C27
\mathchardef\psig="7C20
\mathchardef\zetag="7C10
\mathchardef\epsilong="7C0F
\mathchardef\rhog="7C1A
\mathchardef\taug="7C1C
\mathchardef\upsilong="7C1D
\mathchardef\iotag="7C13
\mathchardef\thetag="7C12
\mathchardef\pig="7C19
\mathchardef\sigmag="7C1B
\mathchardef\etag="7C11
\mathchardef\omegag="7C21
\mathchardef\kappag="7C14
\mathchardef\lambdag="7C15
\mathchardef\mug="7C16
\mathchardef\xig="7C18
\mathchardef\chig="7C1F
\mathchardef\nug="7C17
\mathchardef\varthetag="7C23
\mathchardef\varpig="7C24
\mathchardef\varrhog="7C25
\mathchardef\varsigmag="7C26
\mathchardef\Omegag="7C0A
\mathchardef\Thetag="7C02
\mathchardef\Sigmag="7C06
\mathchardef\Deltag="7C01
\mathchardef\Phig="7C08
\mathchardef\Gammag="7C00
\mathchardef\Psig="7C09
\mathchardef\Lambdag="7C03
\mathchardef\Xig="7C04
\mathchardef\Pig="7C05
\mathchardef\Upsilong="7C07

\newtheorem{thmintro}{Theorem}
\newtheorem{thm}[subsection]{Theorem}
\newtheorem{lem}[subsection]{Lemma}

\newtheorem{prop}[subsection]{Proposition}

\theoremstyle{definition}
\newtheorem{defn}[subsection]{Definition}
\newtheorem{exintro}{Example}

\newtheorem{def-prop}[subsection]{Proposition-Definition}
\newtheorem{def-thm}[subsection]{Theorem-Definition}
\newtheorem{def-lem}[subsection]{Lemma-Definition}

\theoremstyle{remark}

\theoremstyle{plain}

\numberwithin{equation}{subsection}

\def\boxit#1#2{\setbox1=\hbox{\kern#1{#2}\kern#1}
\dimen1=\ht1 \advance\dimen1 by #1
\dimen2=\dp1 \advance\dimen2 by #1
\setbox1=\hbox{\vrule height\dimen1 depth\dimen2\box1\vrule}
\setbox1=\vbox{\hrule\box1\hrule}
\advance\dimen1 by .4pt \ht1=\dimen1
\advance\dimen2 by .4pt \dp1=\dimen2 \box1\relax}

\newcommand{\sq}{\mathrel{\square}}
\newcommand{\ord}{\operatorname{ord}}

\begin{document}

\setcounter{tocdepth}{2} 

\title[Approximations in $p$-adic geometry]{Approximations and Lipschitz continuity in $p$-adic semi-algebraic and subanalytic geometry}

\author{Raf Cluckers}
\address{Universit\'e Lille 1, Laboratoire Painlev\'e, CNRS - UMR 8524, Cit\'e Scientifique, 59655
Villeneuve d'Ascq Cedex, France, and,
Katholieke Universiteit Leuven, Department of Mathematics,
Celestijnenlaan 200B, B-3001 Leu\-ven, Bel\-gium\\}
\email{raf.cluckers@math.univ-lille1.fr}
\urladdr{http://math.univ-lille1.fr/$\sim$cluckers}

\author{Immanuel Halupczok}
\address{Institut f\"ur Mathematische Logik und Grundlagenforschung
Universit\"at M\"unster\\
Einsteinstra\ss{}e 62\\
48149 M\"unster\\
Germany}
\email{math@karimmi.de}
\urladdr{http://www.immi.karimmi.de/en.math.html}

\begin{abstract}
It was already known that a $p$-adic, locally Lipschitz continuous semi-algebraic function is
piecewise Lipschitz continuous, where the pieces can be taken semi-algebraic.
We prove that if the function has locally Lipschitz constant $1$, then it is also
piecewise Lipschitz continuous with the same Lipschitz constant $1$. We do this by proving the following
fine preparation results for $p$-adic semi-algebraic functions in one variable.
Any such function can be well approximated by a monomial with fractional exponent such that moreover
the derivative of the monomial is an approximation of the derivative of the function.
We also prove these results in parametrized versions and in the subanalytic setting.
\end{abstract}

\maketitle

\section*{Introduction}

In several senses, $p$-adic manifolds do not have curvature, they are flat. For example, any $p$-adic analytic submanifold of $\QQ_p^n$ is locally analytically isometric to an open ball in $p$-adic Euclidean space. In this paper we show another phenomenon related to this absence of curvature, namely related to Lipschitz continuity. On the reals, a locally Lipschitz continuous function with constant $\varepsilon$, say on a sufficiently nice domain $X$ in $\RR^n$, can be expected to be globally Lipschitz continuous, but usually this happens only with a bigger Lipschitz constant $C>\varepsilon$; see e.g.\ \cite{Kurdyka} or Example~\ref{exR} below. The increase of the Lipschitz constant when passing from local to global in the reals can happen for example for a function defined on a circle. In the context of semi-algebraic or subanalytic sets on the $p$-adics, one does not need to increase the Lipschitz constant to pass from local to global, but to encompass the total disconnectedness, one has to break the domain into finitely many pieces on each of which the function becomes Lipschitz continuous (with the original Lipschitz constant).
More precisely, we prove:
\begin{thmintro}\label{thm1}
Let $f\colon X\subset \QQ_p^n\to\QQ_p$ be a semi-algebraic function. Suppose that, locally around each point $x\in X$, the function $f$ is Lipschitz continuous with Lipschitz constant $\varepsilon$. Then there exists a finite partition of $X$ into semi-algebraic parts $A_i$ such that each of the restrictions $f_{|A_i}$ is Lipschitz continuous with constant $\varepsilon$ (globally on the part $A_i$). The same property is true when one replaces semi-algebraic by subanalytic.
\end{thmintro}

In fact we prove a parametrized version of Theorem~\ref{thm1}, see Theorem~\ref{rellip}.
Theorem~\ref{thm1} is false if one would replace $\QQ_p$ by $\RR$; see Example~\ref{exR} below.

Let us sketch the $p$-adic context.
Let $K$ be a finite field extension of the field of $p$-adic numbers $\QQ_p$. One of the main results of \cite{CCL} states that if $f\colon X\subset K^n\to K$ is a semi-algebraic (resp.~subanalytic) function which is locally Lipschitz continuous with constant $\varepsilon$, then there exists a finite partition of $X$ into semi-algebraic  (resp.~subanalytic) pieces $A$ such that the restriction $f_{|A}$ of $f$ to any of the pieces is globally Lipschitz continuous with possibly some bigger Lipschitz constant $C$; this is the $p$-adic analogue of one of Kurdyka's results in \cite{Kurdyka}.
Theorem~\ref{thm1} states that one can choose the finite partition in such a way that one can take $C=\varepsilon$. As in \cite{CCL}, this is proved by induction on $n$, and hence, a family (that is, parametrized) version of this Lipschitz continuity result is more natural and flexible to work with, see Theorem~\ref{rellip}.

The main ingredient in \cite{CCL} is a preparation result stating that the domain of a
definable function $f$ can be cut into nice pieces (cells) on which $f$ behaves nicely.
The proof of our Theorem~\ref{rellip} is essentially the same, the main difference being that
we need a finer piecewise preparation result for $f$.
Indeed, the preparation result used in \cite{CCL} is
Proposition~3.12 of \cite{CCL} (which is finer than the more classical cell decomposition results); this proposition in some sense gives a compatible cell decomposition of the domain of
$f\colon X\subset \QQ_p^n\to\QQ_p$ and its image, when $n=1$.
We refine that proposition to get a Preparation Theorem~\ref{thm:simulPrep:basic} which allows one to approximate $f$ in a piecewise way by monomials with fractional exponents such that at the same time the derivative of $f$ is approximated by the derivatives of the monomials, still for $n=1$. In fact, the Preparation Theorem~\ref{thm:simulPrep:basic} treats this property in semi-algebraic, resp.~subanalytic, families.

The real situation is different, as is shown by the following example.
\begin{exintro}\label{exR}
Let $g\colon S^1\to\RR$ be the function on the unit circle in the real plane that sends $z\in S^1$ to the arc length between $z$ and a fixed point on $S^1$. Let $X$ be the open annulus $\{x\in\RR^2\mid 1< |x| <2 \}$ and let $f\colon X\to\RR$ be the function which sends $x\in X$ to $g(z)/r$, where $x=r\cdot z$ with $z\in S^{1}$ and $1 < r < 2$. Then $f$ is globally subanalytic and locally Lipschitz continuous with constant $1$, but there does not exist a finite partition of $X$ into parts on which $f$ is Lipschitz continuous with constant $1$. One is obliged to replace $1$ by $e$ for some real $e>1$ in order to find a finite partition of $X$ into parts on which $f$ is $e$-Lipschitz continuous.
\end{exintro}

\section{Some general definitions}

Let $K$ be a valued field, with (additively written) value group $\Gamma$ and valuation map $\ord\colon K^\times\to\Gamma$. Put the valuation topology on $K^n$ for $n\geq 1$. For $X\subset K$ open, a function $f\colon X\to K$ is called $C^1$ if $f$ is differentiable at each point of $X$ and the derivative $f'\colon X\to K$ of $f$ is continuous (this notion of $C^1$ is more naive than the one of H.~Gl\"ockner \cite{Glock}, but suffices for our purposes).

Write $\cO_K$ for the valuation ring of $K$, with maximal ideal $\cM_K$. Further, write $RV_K$ for the union of the quotient $K^\times/(1+\cM_K)$ with $\{0\}$, and $rv\colon K\to RV_K$ for the natural quotient map $K^\times\to K^\times/(1+\cM_K)$ extended by $rv(0)=0$. Similarly, write $RV_{K,n}$ for the union of the quotient $K^\times/(1+\cM^n_K)$ with $\{0\}$, and $rv_n\colon K\to RV_{K,n}$ for the natural map.

A ball in $K$ is by definition a set of the form
$
\{t \in K\mid \ord ( t - a ) \geq z\}
$
for some $a\in K $ and some $z \in \Gamma$.

\begin{defn}[Jacobian property]\label{defjacprop}
Let $F\colon B_1\to B_2$ be a function with
$B_1,B_2\subset K$. Say that $F$
\textit{has the Jacobian property} if the following
conditions (a) up to (d) hold:
\begin{itemize}

\item[(a)] $F$ is a bijection $B_1\to B_2$ and $B_1$ and $B_2$ are balls;

\item[(b)] $F$ is $C^1$ on $B_1$;

\item[(c)] $\ord (\partial F/\partial x)$ is constant (and finite) on $B_1$;

\item[(d)] for all $x,y\in B_1$ with $x\not=y$, one has
$$
\ord (\partial F/\partial x)+\ord(x-y)=\ord(F(x)-F(y)).
$$
\end{itemize}
\end{defn}

\begin{defn}[$n$-Jacobian property]\label{defnjacprop}
Let $F\colon B_1\to B_2$ be a function with
$B_1,B_2\subset K$ and let $n>0$ be an integer. Say that $F$
\textit{has the $n$-Jacobian property} if $F$ has the Jacobian property,
and moreover the following stronger versions of conditions (c) and (d) hold:
\begin{itemize}
\item[(c')] $rv_n (\partial F/\partial x)$ is constant (and nonzero) on $B_1$;

\item[(d')] for all $x,y\in B_1$ one has
$$
rv_n (\partial F/\partial x)\cdot rv_n(x-y)=rv_n(F(x)-F(y)).
$$
\end{itemize}
\end{defn}

For convenience, $0$-Jacobian property will just mean Jacobian property.

\begin{defn}[Local $n$-Jacobian property]\label{deflocjacprop}
Let $f\colon X\to Y$ be a function with
$X,Y\subset K$ and let $n>0$ be an integer. Say that $f$
\textit{has the local $n$-Jacobian property} if $f$ is a bijection and for each ball $B\subset X$ and for each ball $B'\subset Y$, the restrictions $f_{|B}\colon B\to f_{|B}(B)$ and $(f^{-1})_{|B'}\colon B'\to f^{-1}_{|B'}(B')$ have the $n$-Jacobian property.
\end{defn}

\begin{defn}[Lipschitz continuity]\label{def:lip}
Given two metric spaces $(X, d_X)$ and $(Y, d_Y)$, where $d_X$
denotes the metric on the set $X$ and $d_Y$ the metric on
$Y$, a function $f\colon  X \to Y$ is called Lipschitz continuous if there
exists a real constant $C \geq  0$ such that, for all $x_1$ and
$x_2$ in $X$,
$$
    d_Y(f(x_1), f(x_2)) \leq C\cdot d_X(x_1, x_2).
$$
In the above case, we also call $f$
Lipschitz continuous with Lipschitz constant $C$, or just $C$-Lipschitz continuous. If there is a constant $C$
such that each $x\in X$ has a neighbourhood on which the function $f$ is
$C$-Lipschitz continuous, then $f$ is called locally Lipschitz
continuous with constant $C$, or just locally $C$-Lipschitz
continuous.
\end{defn}

\section{Definable sets over the $p$-adics}\label{defpadic}

From now on,
let $K$ be a finite field extension of $\QQ_p$ with valuation ring
$\cO_K$ and valuation map $\ord\colon K^\times\to\ZZ$. We set $\vert x \vert :=
q_K^{- \ord (x)}$ and $\vert 0 \vert = 0$, where $q_K$ is the cardinality of
the residue field of $K$. Further we choose a uniformizer  $\pi_K$ of $\cO_K$ and for each integer $n>0$, we let $\ac_n\colon K\to \cO_K/(\pi_K^n)$ be the multiplicative map sending $0$ to $0$ and any nonzero $x$ to $x \pi_K^{-\ord(x)}\bmod (\pi_K^n)$.

We recall the notion of (globally) subanalytic subsets of $K^n$
and of semi-algebraic subsets of $K^n$. Let $\cL_{\rm
Mac}=\{0,1,+,-,\cdot,\{P_n\}_{n>0}\}$ be the language of Macintyre and
$\cL_{\rm an}=\cL_{\rm Mac}\cup
\{^{-1},\bigcup_{m>0}K\{x_1,\ldots,x_m\}\}$, where $P_n$ stands for the
set of $n$th powers in $K^\times$,
where $^{-1}$ stands for the field inverse extended to $0$ by $0^{-1}=0$,
where $K\{x_1,\ldots,x_m\}$ is
the ring of restricted power series over $K$ (that is, formal power
series over $K$ converging on
$\cO_K^m$),
 and each element $f$ of $K\{x_1,\ldots,x_m\}$ is interpreted
as the restricted analytic function $K^m\to K$ given by
\begin{equation}
x\mapsto
\begin{cases}
 f(x) & \mbox{if }x\in
\cO_K^m \\
0 & \mbox{else.}
\end{cases}\end{equation}
By subanalytic we mean $\cL_{\rm an}$-definable and by semi-algebraic we mean $\cL_{\rm Mac}$-definable
with parameters from $K$. Note that
semi-algebraic, resp.~subanalytic, sets can be given by a quantifier free formula
with parameters from $K$ in the language $\cL_{\rm Mac}$,
resp.~$\cL_{\rm an}$ by \cite{Macint}, resp.~\cite{DvdD} and \cite{DHM}.

From now on we choose one of the two notions: semi-algebraic or subanalytic, and by \emph{definable} we will mean
semi-algebraic, resp.~subanalytic, according to our fixed choice.

For integers $m>0$ and $n>0$, let $Q_{m,n}$ be the (definable) set
\[
Q_{m,n}:=\{x\in K^\times \mid \ord (x) \in n\ZZ,\ \ac_m(x)=1\}.
\]
For $\lambda\in K$ let
$\lambda  Q_{m,n}$
denote $\{\lambda
x\mid x\in Q_{m,n}\}$. The sets $Q_{m,n}$ are a variant of Macintyre's predicates $P_\ell$ of $\ell$th powers; the corresponding notions of cells are slightly different but equally powerful and similar in usage, since any coset of $P_\ell$ is a finite disjoint union of cosets of some $Q_{m,n}$ and vice versa.

 \begin{defn}[$p$-adic cells]\label{def::cell}
Let $Y$ be a definable set.
A cell $A\subset K\times Y$ over $Y$ is a (nonempty) set of
the form
 \begin{equation}\label{cel}
A=\{(t,y)\in K\times Y\mid y\in Y',\ |\alpha(y)| \sq_1 |t-c(y)| \sq_2
|\beta(y)|,\
  t-c(y)\in \lambda Q_{m,n}\},
\end{equation}
with $Y'$ a definable set,
constants $n>0$, $m>0$, $\lambda$
in $K$, $\alpha,\beta\colon Y'\to K^\times$ and $c\colon Y'\to K$ all
definable functions, and $\sq_i$ either $<$ or no
condition, and such that $A$ projects surjectively onto $Y'\subset Y$.
 We call $c$ the center of the cell $A$, $\lambda Q_{m,n}$ the coset
of $A$, $\alpha$ and $\beta$ the boundaries of $A$, and $Y'$ the
base of $A$. If $\lambda=0$ we call $A$ a $0$-cell
and if $\lambda\not=0$ we call $A$ a $1$-cell.
Call a $1$-cell $A$ unbounded if at least one of the $\sq_i$ in (\ref{cel}) is no condition.
\end{defn}
Sometimes, one additionally requires that $Y'$ is a $K$-analytic manifold
or a cell, but for the purposes of this article, this is not necessary.

\section{The main results}\label{s:results}

In the following two definitions, we introduce
some new notions which will be needed to formulate our version of the Preparation Theorem.
The first notion is that of a ``fractional monomial''---the functions by which we will approximate
arbitrary definable functions. A fractional monomial (``with center $0$'')
is supposed to be something like $t \mapsto et^q$ for some $e \in K$ and $q \in \QQ$.
The following definition makes this precise and moreover allows for parameters.

\begin{defn}[Fractional monomials]\label{fracm}
Let $A\subset K\times Y$ be a cell over $Y$ with center $c$, as in (\ref{cel}). A fractional monomial on $A$ with center $c$ is a continuous, definable function $m\colon A\to K$ such that there exist a definable map $e\colon Y\to K$ and coprime integers $a$ and $b$ with $b>0$ such that for all $(t,y)\in A$
\[
m(t,y)^b = e(y)  (t- c(y))^a.
\]
We use the conventions that $b=1$ whenever $a=0$, that $a=0$ whenever $A$ is a $0$-cell, and that $0^0=1$.
If furthermore $e(y)$ is nonzero for some
 $y$ and some $t$ with $(t,y)\in A$,
 then $a/b$ is independent of any choices and we call $a/b$ the exponent of $m$. In any case we call $e$ the coefficient of $m$.
\end{defn}

Note that, although the center $c$ of a cell $A$ is usually not unique, we assume that cells and fractional monomials on the cells have the same (sometimes implicitly fixed) centers.

Now we define what it should mean for a function to be approximated
by another function, e.g.\ a fractional monomial. This notion only makes sense
on cells, and it only makes sense if both functions ``are compatible'' with that cell;
our notion of $0$-compatibility is essentially the same
as the compatibility notion of Proposition 3.12 of  \cite{CCL}.

\begin{defn}
Let $A\subset K\times Y$ be a $1$-cell over $Y$, let $f\colon A\to K$ be definable, and let $n\geq 0$ be an integer. Write
$$f\times {\rm{id}}\colon A\to K\times Y\colon (t,y)\mapsto (f(t,y),y)$$
 and
$$A_f=(f\times {\rm{id}})(A).
 $$
 Say that $f$ is $n$-compatible with the cell $A$ if either $A_f$ is a $0$-cell over $Y$, or the following holds: $A_f$ is a $1$-cell over $Y$
and for each $y\in Y$, the function $f_y\colon A_y\to f_y(A_y)\colon t\mapsto f(t,y)$ has the local $n$-Jacobian property.

If $g\colon A\to K$ is a second function which is $n$-compatible with the cell $A$ and if we have
\[
A_f=A_g \mbox{ and } rv_n(\frac{\partial f(t,y)}{\partial t}) = rv_n(\frac{\partial g(t,y)}{\partial t}) \mbox{ on }A,
\]
then we say that $f$ and $g$ are $n$-equicompatible with $A$.

If $A'\subset K\times Y$ is a $0$-cell over $Y$ (instead of a $1$-cell), any definable function $h\colon A'\to K$ is said to be $n$-compatible with $A'$, and $h$ and $k\colon A'\to K$ are $n$-equicompatible with $A'$ if and only if $h=k$.
\end{defn}

\begin{thm}[$n$-Preparation Theorem]\label{thm:simulPrep:basic}
Let $X \subset  K\times Y$ and
$f_j\colon  X \to K$ be definable for
$j=1,\ldots,r$ and let $n\geq 0$ be an integer.
Then there exists a finite partition of $X$ into cells $A$ over $Y$ such that the restriction $f_{j|A}$ is $n$-compatible with $A$
for each
cell $A$ over $Y$ and for each $j$, and if one writes $d_j$ for the center of $A_{f_j}$, then there exists a fractional monomial $m_j$ on $A$ such that the functions $d_j+m_j$ and $f_{j|A}$ are $n$-equicompatible with $A$. 

If moreover $A$ is unbounded, then the fractional monomials $m_{j}$ are unique.
\end{thm}

In the above theorem, the center $d_j$ of $A_{f_j}$ is identified with the function on $A_{f_j}$ sending $(t,y)\in A_{f_j}$ to $d_j(y)$.

One can compare our results to the classically known cell decomposition theorem (due to Cohen \cite{cohen}, Denef \cite{D84}, \cite{Dcell}, and the first author \cite{Ccell}).

\begin{thm}[Classical $p$-adic Cell Decomposition]\label{thm:CellDecomp}
Let $X\subset K\times Y$ and $f_j\colon X\to K$ be definable for
$j=1,\ldots,r$ and let $n\geq 1$ be an integer. Then there exist a finite partition of $X$ into
cells $A$ over $Y$ and for each occurring $1$-cell $A$ fractional monomials $m_{j}$ on $A$ such that
 \begin{equation*}
 rv_n(f_j(t,y))=rv_n(m_{j}(t,y))\quad
 \mbox{ for each }(t,y)\in A.
 \end{equation*}
 \end{thm}
 We indicate how Theorem~\ref{thm:CellDecomp} follows from our new Theorem~\ref{thm:simulPrep:basic}. (In fact, also the stronger Proposition 3.12 of \cite{CCL} follows from Theorem~\ref{thm:simulPrep:basic}.) Take a partition of $X$ into cells $A$ as in Theorem~\ref{thm:simulPrep:basic}. By the definition of cells (applied to $A_{f_j}$) we may suppose that either $d_j$ is identically zero or that $rv_n(f_j)=rv_n(d_j)$ on $A$. In the second case one is done since clearly $d_j$ is a fractional monomial on $A$, and in the first case one has $rv_n(f_j)=rv_n(m_j)$ on $A$ and one is also done.
   Proposition 3.12 of \cite{CCL} is as the
   Cell Decomposition
   Theorem~\ref{thm:CellDecomp} with $r=1$ and $f=f_1$ and with
   the extra property that for each $1$-cell $A$ in the partition the restriction $f_{|A}$ is $0$-compatible with $A$.
 In fact, Proposition 3.12 of \cite{CCL} and Theorem~\ref{thm:CellDecomp} will be used to prove Theorem~\ref{thm:simulPrep:basic}.

Theorem~\ref{thm:simulPrep:basic} allows us to improve the main results (Theorems 2.1 and 2.3) of \cite{CCL} to the following.

\begin{thm}[Piecewise Lipschitz continuity]\label{rellip}
Let $\varepsilon>0$ be given. Let $Y$ be a definable set. Let
$f\colon X\subset K^m\times Y\to K$ be a definable function such
that for each $y\in Y$ the function $f(\cdot,y)\colon x\mapsto f(x,y)$ is
locally $\varepsilon$-Lipschitz continuous on $X_y=\{x\mid (x,y)\in X\}$. Then there exists a finite definable partition of $X$ into parts $A_i$ such that
for each $y\in Y$ and $i$ the restriction of $f(\cdot,y)$ to
$A_{iy}$ is (globally on $A_{iy}$) $\varepsilon$-Lipschitz continuous.
\end{thm}
The main point in Theorem~\ref{rellip} is that the constant of the Lipschitz continuity does not change when passing from the local to the piecewise global property.

\section{Proofs of the main results}\label{s:proofs}

We prove some auxiliary results first, after recalling the Banach Fixed Point Theorem in our setting (where we use the discreteness of the $p$-adic valuation to simplify its formulation).
Lemmas \ref{lem:equal-v} and~\ref{lem:equal-rvn} are key points in the proof of Theorem~\ref{thm:simulPrep:basic}.

\begin{lem}[Banach Fixed Point Theorem]\label{lem:banach} \label{lem:Ban}
Suppose that a function  $f\colon \cO_K^n \to \cO_K^n$
is contracting in the sense that
for any $x_1, x_2 \in \cO_K^n$ with $x_1\not=x_2$, $\ord (f(x_1) - f(x_2)) > \ord (x_1 - x_2)  $.
Then $f$ has exactly one fixed point, that is, a point $x\in \cO_K^n$ with $f(x)=x$.
\end{lem}
%
%
%
%
%
%

\begin{lem}\label{lem:equal-v}
Suppose that $B, B_1, B_2 \subset K$ are balls,
that $f_1\colon B \to B_1$ and $f_2\colon B \to B_2$
both satisfy
the Jacobian property, and that $B_1 \cap B_2 \ne \emptyset$.
Suppose moreover $\ord(f_1') \ne \ord(f_2')$. Then
there exists exactly one element $b_0 \in B$ such that $f_1(b_0) = f_2(b_0)$.
\end{lem}

\begin{proof}
By the Jacobian property we may without loss suppose that $\ord(f_1') < \ord(f_2')$ and thus $B_1 \supset B_2$.
Consider the map $f_1^{-1} \circ f_2\colon B \to B$.
By the chain rule for differentiation and the Jacobian property, this map is contracting, and thus, by the Banach Fixed Point Theorem~\ref{lem:Ban}, it has
exaclty one fixed point $b_0$, which is the point we are looking for.
\end{proof}

\begin{lem}\label{lem:equal-rvn} Let $B\subset K$ be a ball containing $0$ and let $f_1, f_2\colon B \iso B$ be definable functions satisfying the following for some integer $n\geq 1$:
\begin{enumerate}
\item\label{1}
both $f_1$ and $f_2$ have the $n$-Jacobian property;
\item\label{2}
$\rv_n (f_1') \ne \rv_n (f_2')$;
\item\label{3}
 $f_1(a)-f_2(a)  \in \cM_K^{n-1}B$ for some $a\in B$.
 \end{enumerate}
Then  there exists exactly one element $b_0 \in B$ such that $f_1(b_0) = f_2(b_0)$.
\end{lem}
\begin{proof}
First we prove uniqueness. Suppose that $f_1(b)=f_2(b)$ for some $b\in B$. Replacing $f_i$ by $t\mapsto f_i( t +b)-f_i(b)$, we may suppose that  $b=0$ and that $f_i(0)=0$ for $i=1,2$. By the $n$-Jacobian property,
$$
\rv_n (f_i(x)) = \rv_n (f'_i) \cdot \rv_n (x)
$$
for all $x$ in $\cO_K$, and hence by (\ref{2}), $f_1(x)\neq f_2(x)$ for nonzero $x$ in $\cO_K$.

Next we prove existence.
Suppose by contradiction that no such $b_0$ exists.
By continuity of $f_1$ and $f_2$, this implies that there exists an upper bound
in $\ZZ$ on the order of $f_1(x)-f_2(x)$ for $x \in \ZZ$; consider the minimal such upper bound $\gamma$ and choose $b\in B$ with  $\ord(f_1(b)-f_2(b))=\gamma$.
In fact, we may suppose $b = a$, and
replacing $f_i$ by $t\mapsto f_i(t +b)$, we may moreover suppose that $b=0$.
Put $c:= f_1(0) - f_2(0)$ and $d:= f_1'(0) - f_2'(0)$. We have, by (\ref{1}),
$$
\ord(f_i(x) - f_i'(0) x - f_i(0) ) \geq \ord(x) + n
$$
for $i=1,2$ and all $x\in B$, and, subtracting,
\begin{equation}\label{4}
\ord(f_1(x) - f_2(x) - dx -c) \geq \ord(x) + n  
\end{equation}
for all $x$ in $B$. Take the unique $b'\in K$ with $ db'-c=0 $.  Since $\ord(d) < n$ by (\ref{1}) and (\ref{2}) and since $\ord(c)=\gamma$, one finds $\ord(b')>\gamma-n$, and thus from (\ref{3}) follows that $b'\in B$. Plugging in $b'$ for $x$ in (\ref{4}) one finds that  $\ord(f_1(b')-f_2(b'))>\gamma$, a contradiction to the choice of $\gamma$.
\end{proof}

We recall a Definition of \cite{CCL}.

\begin{def-prop}[Balls of cells, \cite{CCL}]\label{ballcell}
Let $Y$ be definable. Let $A\subset K\times Y$ be a
$1$-cell over $Y$. 
Then, for each
$(t,y)\in A$, there exists a unique maximal ball
$B_{t,y}$ containing $t$ and satisfying $B_{t,y}\times \{y\}\subset
A$, where
maximality  is under
 inclusion. For fixed $y_0\in Y$ we call the collection of balls
$\{B_{t,y_0}\}_{\{t\mid (t,y_0)\in A\}}$ the balls of the cell $A$
above $y_0$.
\end{def-prop}

\begin{defn}\label{thin}
Call a $1$-cell $A$ over $Y$ thin if the collection of balls of $A$ above any $y\in Y$ consists of at most one ball.
\end{defn}

Theorem~\ref{thm:simulPrep:basic} relies on both Lemmas \ref{lem:equal-v} and~\ref{lem:equal-rvn}. The idea for the applications of these lemmas is that they suppose a certain bad behavior on a ball in view of the desired conclusion of Theorem~\ref{thm:simulPrep:basic}, and then they yield some isolated special points. In our definable set-up, these isolated special points form a discrete definable subset; such sets are known to be finite, and moreover their size is uniformly bounded in families,
hence these finite sets single out a finite collection of balls, each of which corresponds to a thin cell. On these thin cells, we prove Theorem~\ref{thm:simulPrep:basic} separately, directly from the Jacobian property.

\begin{proof}[Proof of Theorem~\ref{thm:simulPrep:basic}]
By Proposition 3.12 of \cite{CCL} we may suppose that for each $j$ separately, there is a finite partition of $X$ into cells over $Y$ such that the restriction of $f_j$ to each cell in the partition is $0$-compatible with that cell. In fact, by using the $n$-Jacobian property of \cite{CLip} instead of only the $0$-Jacobian property, Proposition 3.12 also holds with $n$-compatibility instead of $0$-compatibility, with the same proof as in \cite{CCL}.
Moreover, by piecewise linearity results of Presburger functions on Presburger sets, see e.g.~\cite{pres}, we may suppose that $\ord (\partial f_j(t,y)/\partial t)$ depends linearly on $\ord (t-c(y))$, on each $1$-cell with center $c$ in the $j$-th partition of $X$, and this for each $j$.
The following statement now easily follows from the Cell Criterion 3.8 from \cite{CCL}. For any cell $A$ in the $j$-th partition, and any cell $B$ over $Y$ with center $d$ with  $B\subset A$ there exists a finite partition of $B$ into cells $B_i$, whose centers are the natural restrictions of $d$, and such that $f_{j|B_i}$ is $n$-compatible with $B_i$. We will use this statement five times to simplify the set-up.

Firstly, we may suppose that $X$ is a cell over $Y$ such that all the $f_j$ are $n$-compatible with $X$. If $X$ is a $0$-cell over $Y$, then we are already done (set $m_j := 0$),
so we may suppose that $X$ is a $1$-cell over $Y$. Write $d_j$ for the center of $X_{f_j}$ for each $j$.

Secondly, by the Classical Cell Decomposition Theorem and up to a finite partition of $X$, we may suppose that there exist fractional monomials $m_j$ on $X$ for each $j$ such that
\begin{equation}\label{fjdmj}
\rv_n(f_j-d_j) = \rv_n(m_j)
\end{equation}
holds on $X$.
We may moreover exclude the simple cases that $f_j-d_j$ or $m_j$ are constant.

Thirdly, by the simple form of fractional monomials (namely taking a definable choice out of at most $b$ roots), we may further suppose that also the $m_j$ are $n$-compatible with $X$.

Fourthly, we may suppose for each $j$ that either $X_{f_j}$ is included in $X_{d_j+m_j}$ or, vice versa, $X_{d_j+m_j}$ is included in $X_{f_j}$, by comparing the balls (in the sense of Definition~\ref{ballcell}) of the cells $X_{f_j}$ and $X_{d_j+m_j}$ and possibly modifying
the coefficients $e_j$ of the $m_j$.
Indeed, since $X_{f_j}$ and $X_{d_j+m_j}$ have the same center,
we can choose finitely many $r_{j,\nu} \in K^\times$ such that
for any $y \in Y$ and any ball $B$ of $X$ above $y$, there exists a $\nu$ such that
one of the balls $f(B)$ and $(d_j+r_{j,\nu}m_j)(B)$ is contained in the other one.

Now by Lemma~\ref{lem:equal-v} and fifthly, we can reduce to one of the following two cases: either $X$ is a thin cell, or the set $X_{f_j}$ equals $X_{d_j+m_j}$ for each $j$.
If $X$ is a thin cell, then, by the $n$-Jacobian property, we know that, for each individual $y\in Y$ such that $X_y$ is nonempty, $f_{jy}\colon X_y\to f_{jy}(X_y)\colon t\mapsto f_j(t,y)$ has the $n$-Jacobian property and hence, there exists a linear function $\ell_{jy}\colon X_y\to f_{jy}(X_y)\colon t\mapsto a_{jy}t+b_{jy}$ such that $\ell_{jy}$ and $f_{jy}$ are $n$-equicombatible with the ball $X_y$. By the definability of Skolem functions (also called definability of sections), we may suppose that the $a_{jy}$ and the $b_{jy}$ depend definably on $y$. Hence, we can take as our final $d_j+m_j$ the (linear) $(t,y)\mapsto a_{jy}t+b_{jy}$, that is, the definable functions $d_j(y)=b_{jy}$ on $Y$ and $m_j(t,y)=a_{jy}t$ on $X$ are as desired.

There only remains to treat the case that $X_{f_j}$ equals $X_{d_j+m_j}$ for each $j$ (and $X$ is not thin).
For this case, we will apply Lemma~\ref{lem:equal-rvn} in a similar way as we applied Lemma~\ref{lem:equal-v}
before. More precisely, by Lemma~\ref{lem:equal-rvn} we can exclude finitely many thin cells (each of which can be treated as before), such that
$rv_n(\partial f_j / \partial t) = rv_n(\partial m_j / \partial t)$ holds on the remaining part, and we are done.
For this application of Lemma~\ref{lem:equal-rvn}, it remains to ensure that condition (3)
of that lemma holds. Condition (3) of Lemma~\ref{lem:equal-rvn} can indeed be ensured by further partitioning $X$ if necessary and by
slightly modifying the coefficients $e_j$ of the monomials $m_j$, as we already did in fourthly.
\end{proof}

The following proposition specifies that one can take $C=1$ for the Lipschitz constant in Proposition 2.4 of \cite{CCL}.
\begin{prop}[Cells with $1$-Lipschitz continuous centers]\label{Lipcenter}
Let $Y$ and $X\subset K^m\times Y$ be definable. Then there
exist a finite partition of $X$ into definable parts
$A$ and for each part $A$ a coordinate projection
$$
\pi\colon K^m\times Y\to K^{m-1}\times Y
$$
such that, over $K^{m-1}\times Y$ along this projection $\pi$, the
set $A$ is a cell with center $c\colon \pi(A) \to K$ and such that moreover the
function
 $$
c(\cdot,y)\colon  (t_1,\ldots,t_{m-1})\mapsto
c(t_1,\ldots,t_{m-1},y)
 $$
is $1$-Lipschitz continuous on $\pi(A)_y$ for each $y\in Y$.
\end{prop}

As in \cite{CCL}, we prove Proposition~\ref{Lipcenter} and Theorem~\ref{rellip} by a
joint induction on $m$. More precisely, assuming both
Proposition~\ref{Lipcenter} and Theorem~\ref{thm:simulPrep:basic}
for $\le m - 1$, we first prove Proposition~\ref{Lipcenter} for $m$
and then Theorem~\ref{thm:simulPrep:basic} for $m$.
Since Proposition~\ref{Lipcenter} is trivial for $m = 1$, to start the induction
it suffices to prove Theorem~\ref{thm:simulPrep:basic} for $m = 1$.

\begin{proof}[Proof of Proposition~\ref{Lipcenter} for $m$ using the induction hypothesis.]
This proof is exactly as the proof of Proposition 2.4 in \cite{CCL}, invoking our Theorem~\ref{rellip} for $m-1$ instead of Theorem 2.3 of \cite{CCL} to control the Lipschitz constants, and where, in the second last sentence of the proof of Proposition 2.4 of \cite{CCL}, one replaces ``are bounded in norm'' by ``have norms $\leq 1$'' and one further replaces $C$ by $1$.

Roughly, the idea is the following: start with any cell decomposition of $X$ along any coordinate
projection $\pi$. Consider a cell $A$ with center $c\colon \pi(A) \to K$. If
$|\partial c/\partial t_i| \le 1$ for all $i \le m - 1$, then we are done; otherwise, interchange
the role of $t_i$ and the projection coordinate. We would like to say that $A$ is still
a cell using this different coordinate projection, and that the graph of the center is the same as before.
Of course, this is not true in general, but by cutting $A$ into pieces, this can be easily achieved if $A$ was a $0$-cell and it can be achieved with some work if $A$ was a $1$-cell.
\end{proof}

\begin{proof}[Proof of Theorem~\ref{rellip} for $m=1$]
We are given $\varepsilon>0$, $Y$ a definable set, and
$f\colon X\subset K\times Y\to K$ a definable function such that
for each $y\in Y$ the function $f(\cdot,y)\colon x\mapsto f(x,y)$ is
locally $\varepsilon$-Lipschitz continuous on its natural domain
$X_y:=\{x\in K\mid (x,y)\in X\}$.
 Use Theorem~\ref{thm:simulPrep:basic} to partition $X$ into finitely many cells
$X_i$ over $Y$.
By working piecewise we may suppose that $X=X_1$ and that $X$ and $X_f$ are $1$-cells over $Y$.
By the Jacobian property, $f(\cdot,y)$ is $C^1$ and by local $\varepsilon$-Lipschitz
continuity,
\begin{equation}\label{jace} |\partial f(x,y)/\partial x
|\leq \varepsilon
\end{equation}
for all $(x,y)\in X$. Write $c$ for the center of $X$ and $d$ for the center of $X_f$.
Since a function $g\colon A\subset K\to K$ is $\varepsilon$-Lipschitz continuous if
and only if $A\to K\colon x\mapsto g(x+a)+b$ is $\varepsilon$-Lipschitz continuous
for any constants $a,b\in K$, we may suppose, after translating, that $c$
and $d$ are identically zero.  Thus, we have that $f$ and $m$ are $n$-equicompatible with $X$ for a certain fractional monomial $m$.

Now fix $y\in Y$. Take $(x_1,y)$ and $(x_2,y)$ in $X$. If $x_1$ and
$x_2$ both lie in the same ball $B_{x_1,y}$ above $y$, then
\begin{equation}\label{m1jac}
|( \partial f(x_1,y)/\partial x)  \cdot (x_1-x_2)|= | f(x_1,y) -
f(x_2,y) |
\end{equation}
by the Jacobian property and we are done by (\ref{jace}).

Next suppose that $B_{x_1,y}$ and $B_{x_2,y}$  are two different
balls. By our assumption that $c$ and $d$ are identically zero, we can
write
\begin{equation}\label{b1}
 B_{x_i,y} = \{x\in K\mid \ord (x)=a_{x_i,y},\ \ac_\ell(x)=\ac_\ell\lambda \}
\end{equation}
for $i=1,2$ and some $\ell$ and $a_{x_i,y}$,
and likewise for their images under $f(\cdot,y)$ for $i=1,2$:
\begin{equation}\label{b2}
\big(f(\cdot,y)\big)(B_{x_i,y}) = 
\{z\in K \mid \ord (z)=b_{x_i,y},\ \ac_{\ell'}(z)=\ac_{\ell'}\mu
 \}.
\end{equation}
From these descriptions we get the inequalities:
\[
\ord(f(x_1,y)- f(x_2,y)) = \min_{i=1,2} (b_{x_i,y})
\]
and
\[
\ord(x_1-x_2) = \min_{i=1,2} (a_{x_i,y} ).
\]
On the other hand by the
Jacobian
property (d) one finds, by comparing the sizes of the balls (\ref{b1}) and (\ref{b2}),
\[
\ell +  \ord (\partial f(x_i,y)/\partial x)  + a_{x_i,y}  =  \ell' +
b_{x_i,y}
\]
for $i=1,2$. Putting this together with (\ref{jace}) yields:
\begin{equation}\label{c1}
 |f(x_1,y)- f(x_2,y)| = \max_{i=1,2} q_K^{-b_{x_i,y}}
 \leq   \varepsilon  q_K^{\ell'-\ell}   \max_{i=1,2} q_K^{-a_{x_i,y}}  =  \varepsilon q_K^{\ell'-\ell} |x_1-x_2|.
\end{equation}
If $\ell' - \ell\leq 0$ then we are done by (\ref{c1}).
Also if the exponent of the fractional monomial $m$ is $1$, then, by the $n$-equicompatibility of $m$ and $f$ and the linearity of $m$, one must have $\ell=\ell'$ and the statement follows from (\ref{c1}).
Finally suppose that the exponent of $m$ is unequal to $1$ and that $\ell' - \ell > 0$.
Then $|\partial f(x,y)/\partial x
|$ is not constant on $X$, since it is equal to $|\partial m(x,y)/\partial x
|$. Hence, excluding finitely many thin cells from $X$, it follows that we can suppose that also
\begin{equation}\label{jace2}
 |\partial f(x,y)/\partial x
|\leq \varepsilon q_K^{-(\ell'-\ell)}
\end{equation}
holds for all $(x,y)\in X$. Now we are done by a similar calculation as in (\ref{c1}), using (\ref{jace2}) instead of (\ref{jace}). Each of the remaining thin cells can be treated as separate part and the statement on such a part follows again by (\ref{m1jac}) and (\ref{jace}).
\end{proof}
\begin{proof}[Proof of Theorem~\ref{rellip} for general $m>1$]
This proof is the same as the proof of Theorem 2.3 for general $m>1$ of \cite{CCL}; we indicate the changes to be made. (The changes are all related to the Lipschitz constants.) One should invoke our Proposition~\ref{Lipcenter} for $m-1$ instead of Proposition 2.4 of \cite{CCL} and then use $C=1$. Then, in the second sentence after (4.1.2) of \cite{CCL}, one should note that the bi-Lipschitz transformation which replaces $x_m$ by $x_m-c (\hat x,y)$ but which preserves the other coordinates, is an isometry and from the statement ($*$) of the proof in \cite{CCL} on, one can take $C=\varepsilon$.
Further the proof is the same as the proof of Theorem 2.3 in \cite{CCL}.
In a nutshell, the idea in \cite{CCL} is to partition the domain into cells in which one can move along lower-dimensional subsets for which the induction hypothesis can be used.
\end{proof}

\bibliographystyle{amsplain}

\end{document}